\documentclass[a4paper, reqno, 11pt]{amsart}

\usepackage{amsmath}
\usepackage{amssymb}
\usepackage{amsfonts}

\def \r{\rho}

\newcommand{\ba}{\begin{array}}
\newcommand{\ea}{\end{array}}

\newcommand{\ben}{\begin{equation}}
\newcommand{\een}{\end{equation}}

\newcommand{\bean}{\begin{eqnarray}}
\newcommand{\eean}{\end{eqnarray}}

\newcommand{\bea}{\begin{eqnarray*}}
\newcommand{\eea}{\end{eqnarray*}}

\newcommand{\bd}{\begin{description}}
\newcommand{\ed}{\end{description}}

\newcommand{\bc}{\begin{center}}
\newcommand{\ec}{\end{center}}

\newcommand{\be}{\begin{enumerate}}
\newcommand{\ee}{\end{enumerate}}

\newcommand{\ds}{\displaystyle}

\newtheorem{theo}{Theorem}[section]

\newtheorem{lem}[theo]{Lemma}
\newtheorem{rem}[theo]{Remark}

\author{Najoua El Ghani}
\address{Universit\'e de Sousse, Institut Sup\'erieur d'Informatique et des Technologies de Communication, D\'epartement de R\'eseaux et Multim\'edia, Rue G.P.1 Hammam Sousse-4011, Tunisie} \email{\sl 	 Najoua.ElGhani@fst.rnu.tn}

\author{Mohamed Majdoub}
\address{Universit\'e de Tunis El Manar, Facult\'e des Sciences de Tunis,  Laboratoire \'equations aux d\'eriv\'ees partielles (LR03ES04), 2092 Tunis, Tunisie} \email{\sl mohamed.majdoub@fst.rnu.tn}
\thanks{The authors are grateful to the Laboratory of
PDE at the Faculty of Sciences of Tunis.}

\title[Global well posedness for the DD-M system]
{Global well posedness for the drift-diffusion-Maxwell system in 2D}
\date{\today}

\begin{document}

\begin{abstract}
We prove global existence of strong solutions to the drift-diffusion-Maxwell system in two space dimension. We also provide an exponential growth estimate for the $H^1$ norm of the solution.
\end{abstract}





\maketitle



\section{Introduction}


We consider a coupled system of equations consisting of the equation of the charge and current density and
 Maxwell's equations of electromagnetism, the coupling comes from the Lorentz force.

\subsection{The model}


We consider the Drift-Diffusion-Maxwell system (DD-M) for short, namely:
\begin{equation}\label{s1}\left\{%
\begin{array}{lll}
\partial_t\rho+div_xj =0,& \; & \mbox{in }(0,\;T)\times\mathbb{R}^2,\\
\partial_{t}E-curl_{x}B=-j,& \; & \mbox{in }(0,\;T)\times\mathbb{R}^2,\\
\partial_{t}B+curl_{x}E=0,& \; & \mbox{in }(0,\;T)\times\mathbb{R}^2,\\
div_{x}E=\rho,& \; & \mbox{in }(0,\;T)\times\mathbb{R}^2,\\
div_{x}B=0,& \; & \mbox{in }(0,\;T)\times\mathbb{R}^2,\\
j= \rho E-\nabla_x\rho,& \; & \\
\end{array}%
\right.\end{equation}
where $E,\,B$ are the electric and magnetic fields and $\rho,\, j$
are respectively the charge and current densities. We also supplement (\ref{s1}) with the following initial data
\begin{equation}\label{d1}\rho(t=0)=\rho_0,\,\,B(t=0)=B_0,\,\,E(t=0)=E_0.\end{equation}
 Here, $\rho,\,B,\,E$ are defined on $\mathbb{R}^{2}$ and take their values in $\mathbb{R}^{3}$, i.e., $E=(E_{1}(t,x),E_{2}(t,x),E_{3}(t,x))$, $B=(B_{1}(t,x),B_{2}(t,x),B_{3}(t,x))$, $\rho=\rho(t,x)$ for any $(t,x)\in(0,\;T)\times\mathbb{R}^{2}$. The notation $curl_{x}B$ corresponds to
\begin{equation}\label{x}
\nabla\wedge B =
\left(\begin{array}{c}
  \partial_1 \\
  \partial_2  \\
0
\end{array}\right)\wedge\left(\begin{array}{c}
  B_1 \\
 B_2  \\
B_3
\end{array}\right)
\end{equation}
The system (\ref{s1}) has the following energy identity:\begin{eqnarray}\label{e}
\frac{1}{2}\partial_{t}[\|\rho\|^{2}_{L^{2}}+\|E\|^{2}_{L^{2}}+\|B\|^{2}_{L^{2}}]+\|\nabla\rho\|^{2}_{L^{2}}&\leq&0\end{eqnarray}
which is similar to the energy identity for the Maxwell-Navier-Stokes system used in \cite{Mas}.

Before stating our main result, let us mention that the Drift-Diffusion-Maxwell model (\ref{s1}) is derived from a Vlasov-Maxwell-Fokker-Planck system \cite{gh1} which is motivated from plasma physics and  Drift-Diffusion models can also be derived from other singular limits. We refer for instance to \cite{gh} where the Drift-Diffusion-Poisson model is derived from a Vlasov-Poisson-Fokker-Planck system. Let us recall that the Drift-Diffusion model is a standard model for semiconductors physics and suited for numerical computations we refer to \cite{A T2004,gaj,m-t} for a discussion about this model. Here, we would like to explain a little bit the relevance of the model. The first equation in (\ref{s1}) is the mass conservation equation (Continuity Equation). The second equation is the Ampere-Maxwell equation which includes here the displacement current $\partial_{t}E$. The third equation of (\ref{s1}) is the Faraday's law and finally, the forth and fifth equation are the Gauss's law (electric and magnetic).
\subsection{Statement of the result}
We want to prove in 2 space dimension, the global existence of solutions such that $\rho_0\in L^{\infty}(\mathbb{ R}^{2})\cap H^{1}(\mathbb{R}^{2})$ and $B_0,\,\,E_0\in H^{1}(\mathbb{R}^{2}).$ The proof uses the conservation of the energy as well as a logarithmic estimate to bound the $L^{\infty}$ norm of $\rho$ in terms of the $H^{1}$ norm of $\rho$.

 Our main result is the following:
\begin{theo}Take $\rho_0\in L^{\infty}(\mathbb{ R}^{2})\cap H^{1}(\mathbb{R}^{2})$ and $B_0,\,\,E_0\in H^{1}(\mathbb{R}^{2}).$ Then, there exists a unique global solution $(\rho, E, B)$ of (\ref{s1}) such that for all $T>0$, $\rho\in C([0, T); L^{2})\cap L^{2}(0, T; H^{1})$ and $E$, $B\in C([0, T); H^{1})$. Moreover, $j\in L^{2}(0, T; L^{2})\cap L^{2}(0, T; H^{1})$ and $\rho\in L^{1}(0, T; H^{2})$. In addition, the energy identity (\ref{e}) holds and we have the following exponential growth estimate for all $t>0$:\\
\begin{equation}\label{grow}\|\rho\|_{L^{1}(0, t; H^{1})}+\|(E, B)(t)\|_{H^{1}}\leq (1+\|(E_{0}, B_{0})\|_{H^{1}})e^{C_{0}(t+1)},\end{equation}
where $C_{0}=C[\|\rho_0\|_{L^{2}}^{2}+\|B_0\|_{L^{2}}^{2}+\|E_0\|_{L^{2}}^{2}+1]$ for some constant $C$.
\end{theo}

In the next Section 1.3, we give preliminaries about some regularity estimates. In Section 2, we prove some a priori estimate and derive the growth bound (\ref{grow}). In Section 3, we prove Theorem 1.1 by using a Galerkin approximation.
\subsection{ Preliminaries}

The system (\ref{s1}) has the following energy identity:
$$\frac{1}{2}\partial_{t}[\|\rho\|^{2}_{L^{2}}+\|E\|^{2}_{L^{2}}+\|B\|^{2}_{L^{2}}]+\|\nabla\rho\|^{2}_{L^{2}}\leq0$$
Indeed, multiplying the first equation of (\ref{s1}) by $\rho$, the second one by $E$ and the third one by $B$ and integrating by parts, the energy estimate reads
$$\frac{1}{2}\partial_{t}[\|\rho\|^{2}_{L^{2}}+\|E\|^{2}_{L^{2}}+\|B\|^{2}_{L^{2}}]=$$
$$-\int\rho\nabla.j dx+\int (curl_{x}(B).E -curl_{x}(E).B)dx+\int\nabla\rho.E dx-\int\rho|E|^{2} dx.$$
Since, $j= \rho E-\nabla_x\rho$ and $\nabla.j=\rho div E+ \nabla\rho.E-\triangle\rho$
Then, $\ds\int\rho\nabla.j dx=\int\rho^{2}\underbrace{ div E}_{\rho} dx+ \int\rho\nabla\rho.E dx-\int\rho\triangle\rho dx$
Moreover, $\ds\int (curl_{x}(B).E-curl_{x}(E).B) dx=0$
We obtain,
$$\frac{1}{2}\partial_{t}[\|\rho\|^{2}_{L^{2}}+\|E\|^{2}_{L^{2}}+\|B\|^{2}_{L^{2}}]+\|\nabla\rho\|^{2}_{L^{2}}=$$
$$\underbrace{-\int\rho E \nabla\rho dx}_{I_{1}}\underbrace{-\int\rho^{3}dx}_{I_{2}}\underbrace{-\int\rho E^{2}dx}_{I_{3}}+\underbrace{\int E \nabla\rho dx}_{I_{4}}.$$
We have $$I_{1}=-\int E \nabla(\frac{\rho^{2}}{2}) dx=\int div E \,\frac{\rho^{2}}{2} dx=\frac{1}{2}\int\rho^{3}dx.$$
So, $I_{1}+I_{2}=-\displaystyle\frac{1}{2}\int\rho^{3}dx\leq0$, $I_{3}\leq0$
and $$I_{4}=-\int div E \,\rho dx=-\int\rho^{2}dx\leq0.$$
which yields the desired energy estimate.
\begin{rem}
The energy identity given here is not sufficient to deduce that $E, B\in L^{\infty}(0, T;\;H^{1})$
\end{rem}
We will use the following lemma giving regularity result for the Maxwell equation:
\begin{lem}\cite{Mas}
\;If $(E,\,B)$ solves
\begin{equation}\label{s}\left\{%
\begin{array}{lll}
\partial_{t}E-curl_{x}B=-j,& \; & \;\\
\partial_{t}B+curl_{x}E=0,& \; & \;\\
E(t=0)=E_{0},\,\,B(t=0)=B_{0}& \; & \;\\
\end{array}%
\right.\end{equation}
On some time interval $(0, T)$ then, we have
$$\|(E,\,B)\|_{C([0, T);\;H^{1})}\leq\|(E_{0},\,B_{0})\|_{H^{1}}+\|j\|_{L^{1}(0, T;\;H^{1})}$$
\end{lem}

We can use for the proof of this lemma the Duhamel formula and write $F=e^{tL}F_{0}+\ds\int_{0}^{t}e^{(t-s)L}f(s)ds$
where $F=(E,\;B)$, $L$ is the operator define by $L(E,\;B)=(curl\; B,\;-curl\; E)$ and $f(s)=(j(s),\;0)$. It is then clear that $e^{tL}$ defines an isometry on $H^{s}$ and hence the claim follows.

We refer to \cite{gh} for the proof of the following result which give a regularity of the density:
\begin{lem}
\;Let $\r$ be a positive function such that $\r\in
L^\infty(0,T;\; L^1(\mathbb{R}^2))$, satisfying
\begin{equation}\label{smax}\left\{%
\begin{array}{lll}
\nabla_x\sqrt{\rho}-\frac{1}{2}E\sqrt{\rho} =H\in L^2(0,T;\; L^2(\mathbb{R}^2)),& \; & \;\\
div_x E=\rho.&\; & \;\\
E\in L^\infty(0,T;\; L^2(\mathbb{R}^d)),& \; & \;\\
\end{array}%
\right.\end{equation} then  $$\r\in L^2(0,T;\;
L^2(\mathbb{R}^2)),\hspace*{1cm}E\sqrt{\r}\in L^2(0,T;\;
L^2(\mathbb{R}^2)),$$ and $$\sqrt{\r}\in L^2(0,T;\;
H^1(\mathbb{R}^2)).$$
\end{lem}

We recall here the Littlewood-Paley decomposition of a function. We define $\mathcal{C}$ to be the ring of center $0$, of small radius $1/2$ and great radius $2$. There exist
two nonnegative radial functions $\chi$ and $\varphi$ belonging respectively to $\mathcal{D}(B(0, 1))$ and to $\mathcal{D}(\mathcal{C})$ so that $$\chi(\xi)+\sum_{q\geq0}\varphi(\xi)=0$$
$$|p-q|\geq2\,\Rightarrow\,\mbox{Supp }\varphi(2^{-q}.)\cap \mbox{Supp }\chi(2^{-p}.)=\emptyset.$$
For instance, one can take $\chi\in \mathcal{D}(B(0, 1))$ such that $\chi\equiv 1$ on $B(0, 1/2)$ and take,
$$\varphi(\xi)=\chi(\frac{\xi}{2})-\chi(\xi).$$
Then, we are able to define the Littlewood-Paley decomposition. Let us denote by $\mathcal{F}$ the
Fourier transform on $\mathbb{R}^{d}$. Let $h$, $\widetilde{h}$, $\triangle_{q}$, $S_{q}$ $(q \in \mathbb{Z})$ be defined as follows:
$$h=\mathcal{F}^{-1}\varphi \mbox{ and } \widetilde{h}=\mathcal{F}^{-1}\chi,$$
$$\triangle_{q}u=\mathcal{F}^{-1}(\varphi(2^{-q}\xi)\mathcal{F}u)=2^{qd}\int h(2^{q}y)u(x-y)dy,$$
$$S_{q}u=\mathcal{F}^{-1}(\chi(2^{-q}\xi)\mathcal{F}u)=2^{qd}\int \widetilde{h}(2^{q}y)u(x-y)dy.$$
We point out that $S_{q}u=\sum_{q'\leq q-1,\, q'\in \mathbb{Z}}\triangle_{q'}u$.
\begin{lem}\label{lem1}If $\rho \in  L^2(0,T;\;H^2(\mathbb{R}^2))$, then we have:
$$\|\rho\|_{L^{1}_{t}\,\,L^{\infty}_{x}}\leq C_{0}^{1/2}T^{1/2}+C T^{1/2}\|\nabla\rho\|_{L^{2}_{t,x}}\log(e+\frac{\|\nabla^{2}\rho\|_{L^{2}_{t,x}}}{\|\nabla\rho\|_{L^{2}_{t,x}}})$$
\end{lem}
\begin{proof}
For the proof of this lemma, we use the Littlewood-Paley decomposition of a function. When dealing with functions which depend on $t$ and $x$, the Littlewood-Paley decomposition will only apply to the $x$ variable.

We have $$\rho=S_{1}\rho+\sum_{q=1}^{q=N}\triangle_{q}\rho+\sum_{q\geq N+1}\triangle_{q}\rho$$
when $N$ is the integer. Hence, we have
\bea\label{sds} \|\rho\|_{L^{\infty}}&\leq&  \|S_{1}\rho\|_{L^{\infty}}+\displaystyle\sum_{q=1}^{q=N}\|\triangle_{q}\rho\|_{L^{\infty}}+\sum_{q\geq N}\|\triangle_{q}\rho\|_{L^{\infty}}\nonumber\\
\,&\leq&  \|S_{1}\rho\|_{L^{\infty}}+\displaystyle\sum_{q=1}^{q=N}2^{-q}\|\nabla(\triangle_{q}\rho)\|_{L^{\infty}}+\sum_{q\geq N}2^{-2q}\|\nabla^{2}(\triangle_{q}\rho)\|_{L^{\infty}}\nonumber\\
\,&\leq& C\|\rho\|_{L^{2}}+C\displaystyle\sum_{q=1}^{q=N}\|\triangle_{q}(\nabla\rho)\|_{L^{2}}+C\sum_{q\geq N}2^{-q}\|\triangle_{q}(\nabla^{2}\rho)\|_{L^{2}}\nonumber\\
\,&\leq& C\|\rho\|_{L^{2}}+C\displaystyle\sum_{q=1}^{q=N}\|\nabla\rho\|_{L^{2}}+C\sum_{q\geq N}2^{-q}\|\nabla^{2}\rho\|_{L^{2}}\nonumber\\
&\leq& C\|\rho\|_{L^{2}}+CN\|\nabla\rho\|_{L^{2}}+C2^{-N}\|\nabla^{2}\rho\|_{L^{2}}\eea

Then, by using the Cauchy-Schwartz inequality, we get \\
\bea \displaystyle\int_{0}^{T}\|\rho\|_{L^{\infty}}dt&\leq&  C\|\rho\|_{L^{1}L^{2}}+CN\ds\int_{0}^{T}\|\nabla\rho\|_{L^{2}}dt+C2^{-N}\ds\int_{0}^{T}\|\nabla^{2}\rho\|_{L^{2}}dt\\
\,&\leq& C_{0}^{1/2}T^{1/2}+C\sqrt{T}N\|\nabla\rho\|_{L^{2}_{t,x}}+C\sqrt{T}2^{-N}\|\nabla^{2}\rho\|_{L^{2}_{t,x}}\eea
we optimize in $N$, by taking $N$ of the order $\frac{1}{\log(2)}\log(e+\frac{\|\nabla^{2}\rho\|_{L^{2}_{t,x}}}{\|\nabla\rho\|_{L^{2}_{t,x}}})$. Hence,
\bea\|\rho\|_{L^{1}_{t}\,\,L^{\infty}_{x}}\leq C_{0}^{1/2}T^{1/2}+C T^{1/2} \|\nabla\rho\|_{L^{2}_{t,x}}\log(e+\frac{\|\nabla^{2}\rho\|_{L^{2}_{t,x}}}{\|\nabla\rho\|_{L^{2}_{t,x}}})\eea
\end{proof}
\begin{lem}\label{rl1}
There exists a constant $C$ and $C_{0}$, such that for all $T>0$, we have
\begin{equation}\label{sx1}
\|\rho\|_{L^{1}H^{1}}\leq C (e+\|(E_{0}, B_{0})\|_{H^{1}})e^{C_{0}T}.
\end{equation}
\end{lem}
\begin{proof}
We have
\bea \|\rho\|_{L^{2}H^{1}}&\leq&   T^{1/2}\|\rho\|_{L^{\infty}L^{2}}+\|\nabla\rho\|_{L^{2}L^{2}}\\
\,&\leq& C_{0}^{1/2}(T^{1/2}+1)\eea
So, we get
\bea \|\rho\|_{L^{1}H^{1}}&\leq&   T^{1/2}\|\rho\|_{L^{2}H^{1}}\\
\,&\leq&   C_{0}^{1/2}T^{1/2}(T^{1/2}+1)\\
\;&\leq&C e^{C_{0}T}.\eea
\end{proof}
Finally, we recall a 2D Gagliardo-Nirenberg estimate and some classical inequalities.
\begin{lem} For any $u\in H^{1}(\mathbb{R}^{2}))$, we have
\begin{equation}
\label{GN}
\|u\|_{L^{4}}\leq \|u\|^{\frac{1}{2}}_{L^{2}}\|\nabla u\|^{\frac{1}{2}}_{L^{2}}.
\end{equation}
\end{lem}
\begin{lem}
Let $a\geq 1$ and $b\geq 0$ be two real numbers. Then
\begin{equation}
\label{ineq1}
\sqrt{a+b}\leq \sqrt{a}+b,
\end{equation}
and
\begin{equation}
\label{ineq2}
\log(a+b)\leq \log(a)+b.
\end{equation}
\end{lem}

\section{ A priori estimates}


The system (\ref{s1}) has the follwing energy identity
$$\frac{1}{2}\partial_{t}[\|\rho\|^{2}_{L^{2}}+\|E\|^{2}_{L^{2}}+\|B\|^{2}_{L^{2}}]+\|\nabla\rho\|^{2}_{L^{2}}\leq0$$
and hence
 \begin{eqnarray}\label{ident}\frac{1}{2}[\|\rho\|^{2}_{L^{2}}+\|E\|^{2}_{L^{2}}+\|B\|^{2}_{L^{2}}](t)
 +\int_{0}^{t}\|\nabla\rho\|^{2}_{L^{2}}\leq\\\nonumber
\frac{1}{2}[\|\rho_{0}\|^{2}_{L^{2}}+\|E_{0}\|^{2}_{L^{2}}+\|B_{0}\|^{2}_{L^{2}}]\leq C_{0}. \end{eqnarray}
This formally yields the bounds
$$\rho\in L^\infty(0,T;\; L^2)\cap L^2(0,T;\;
H^1),\,E,\,B\in L^\infty(0,T;\; L^2).$$
Here and below $C_{0}$ will denote any constant of the form $C[\|\rho_{0}\|^{2}_{L^{2}}+\|E_{0}\|^{2}_{L^{2}}+\|B_{0}\|^{2}_{L^{2}}+1]$
where $C$ may change from one line to the other.

Moreover, the first equation of (\ref{s1}) can be written as 
$$\partial_{t}\rho-\triangle\rho=-E.\nabla\rho-\rho^{2}.$$
Applying the operator $\nabla$ to this equation
$$\partial_{t}\nabla\rho-\triangle\nabla\rho=-\nabla E.\nabla\rho-E.\nabla^{2}\rho-2\rho\nabla\rho.$$
Multiplying the result by $\nabla\rho$, we obtain
$$\partial_{t}\frac{|\nabla\rho|^{2}}{2}-\nabla\rho\triangle\nabla\rho=-(\nabla E.\nabla\rho)\nabla\rho-(E.\nabla\nabla\rho)\nabla\rho-2\rho|\nabla\rho|^{2}$$
Now, integrating it with respect to $x$
\begin{equation}\label{swx}\frac{1}{2}\partial_{t}\int|\nabla\rho|^{2}dx+\int|\nabla^{2}\rho|^{2}dx=-\int(\nabla E.\nabla\rho)\nabla\rho dx\end{equation} $$-\int(E.\nabla^{2}\rho)\nabla\rho dx-2\int\rho|\nabla\rho|^{2}dx.$$
Applying the operator $\nabla$ to the second equation of (\ref{s1})
$$\partial_{t}\nabla E-curl_{x}\nabla B=\nabla^{2}\rho-\nabla\rho E-\rho\nabla E.$$
Multiplying the result by $\nabla E$ and integrating it with respect to $x$

\begin{equation}\label{sw}\frac{1}{2}\partial_{t}\int|\nabla E|^{2}dx-\int\nabla E curl_{x}\nabla Bdx=\int\nabla^{2} \rho\nabla Edx\end{equation} $$-\int \nabla\rho E\nabla E dx-\int\rho|\nabla E|^{2}dx.$$
 Applying the operator $\nabla$ to the third equation of (\ref{s1})
$$\partial_{t}\nabla B+curl_{x}\nabla E=0$$
Multiplying the result by $\nabla B$ and integrating it with respect to $x$

\begin{equation}\label{sww}\frac{1}{2}\partial_{t}\int|\nabla B|^{2}dx+\int\nabla B\,\, curl_{x}\nabla E\,dx=0.\end{equation}
 Combining (\ref{swx}),(\ref{sw}) and (\ref{sww}), we deduce that

$$\frac{1}{2}\partial_{t}\underbrace{\int(|\nabla\rho|^{2}+|\nabla E|^{2}+|\nabla B|^{2})dx}_{\|\nabla F\|^{2}_{L^{2}}}+\int|\nabla^{2}\rho|^{2}dx=\underbrace{\int\nabla^{2} \rho\nabla Edx}_{I_{1}}$$$$\underbrace{-\int \nabla\rho E\nabla E dx}_{I_{2}}\underbrace{-\int\rho|\nabla E|^{2}dx}_{I_{3}}\underbrace{-\int(\nabla E.\nabla\rho)\nabla\rho dx}_{I_{4}}$$$$\underbrace{-\int(E.\nabla^{2}\rho)\nabla\rho dx}_{I_{5}}\underbrace{-2\int\rho|\nabla\rho|^{2}dx}_{I_{6}}.$$
 Now, using the Young's inequality $(a b\leq\alpha a^{2}+\frac{1}{4\alpha}b^{2}, \forall \alpha>0)$, we can see that
\begin{eqnarray*}I_{1}=\int\nabla^{2}\rho\nabla Edx&\leq&\|\nabla E\|_{L^{2}}\|\nabla^{ 2 }\rho\|_{L^{2}}\\
\,&\leq&2\|\nabla F\|^{2}_{L^{2}}+\frac{1}{8}\|\nabla^{ 2 }\rho\|^{2}_{L^{2}}\end{eqnarray*}
The Gagliardo-Nirenberg inequality \eqref{GN}, gives
\begin{eqnarray*}I_{2}=-\int \nabla\rho E\nabla E dx&=&-\int \nabla\rho \nabla \frac{E^{2}}{2} dx\\&=&\int \nabla^{2}\rho \frac{E^{2}}{2} dx\\\,&\leq&\frac{1}{2}\|\nabla^{2} \rho\|_{L^{2}}\| E\|_{L^{4}}^{2}\\\,&\leq&\frac{1}{2}\|\nabla^{2}\rho\|^{2}_{L^{2}}\|E\|_{L^{2}}\|\nabla E\|_{L^{2}}\\
\,&\leq&\frac{1}{8}\|\nabla^{2}\rho\|^{2}_{L^{2}}+\frac{1}{2}\|E\|^{2}_{L^{2}}\|\nabla F\|^{2}_{L^{2}}\end{eqnarray*}
\begin{eqnarray*}I_{4}=-\int(\nabla E.\nabla\rho)\nabla\rho dx&\leq&\|\nabla E\|_{L^{2}}\|\nabla \rho\|_{L^{4}}^{2}\\\;&\leq&\|\nabla E\|_{L^{2}}\|\nabla \rho\|_{L^{2}}\|\nabla^{2} \rho\|_{L^{2}}\\\;&\leq&\frac{1}{8}\|\nabla^{2} \rho\|^{2}_{L^{2}}+2\|\nabla \rho\|^{2}_{L^{2}}\|\nabla F\|^{2}_{L^{2}}\end{eqnarray*}
\begin{eqnarray*}I_{5}=-\int(E.\nabla\nabla\rho)\nabla\rho dx&=&-\int E.\nabla(\frac{|\nabla\rho|^{2}}{2}) dx\\\;&=&\int\rho\frac{|\nabla\rho|^{2}}{2} dx\\\;&\leq&\frac{1}{2}\|\rho\|_{L^{2}}\| \nabla\rho\|_{L^{4}}^{2}\\\;&\leq&\frac{1}{2}\|\rho\|_{L^{2}}\|\nabla \rho\|_{L^{2}}\|\nabla^{2} \rho\|_{L^{2}}\\\;&\leq& \frac{1}{8}\|\nabla^{2} \rho\|^{2}_{L^{2}}+\frac{1}{2}\| \rho\|^{2}_{L^{2}}\|\nabla\rho\|^{2}_{L^{2}}\end{eqnarray*}
Combining $I_{1}, I_{2}, I_{3}, I_{4}, I_{5}$ and $I_{6}$, we deduce that
\begin{eqnarray*}\frac{1}{2}\partial_{t}\|\nabla F\|^{2}_{L^{2}}+\|\nabla^{2}\rho\|^{2}_{L^{2}}&\leq&\frac{1}{2}\|\nabla^{2}\rho\|^{2}_{L^{2}}+\frac{1}{2}\| \rho\|^{2}_{L^{2}}\|\nabla\rho\|^{2}_{L^{2}}\\\;&\;&+\Big[2+\frac{1}{2}\|E\|^{2}_{L^{2}}+2\|\nabla\rho\|^{2}_{L^{2}}\Big]\|\nabla F\|^{2}_{L^{2}}\end{eqnarray*}
 Then,
\begin{eqnarray*}\partial_{t}\|\nabla F\|^{2}_{L^{2}}+\|\nabla^{2}\rho\|^{2}_{L^{2}}&\leq&\| \rho\|^{2}_{L^{2}}\|\nabla\rho\|^{2}_{L^{2}}\\\;&\;&+\Big[4+\|E\|^{2}_{L^{2}}+4\|\nabla\rho\|^{2}_{L^{2}}\Big]\|\nabla F\|^{2}_{L^{2}}\end{eqnarray*}
 Hence,
\begin{eqnarray}\label{eq}\|\nabla F\|^{2}_{L^{2}}(t)+\int_{0}^{t}\|\nabla^{2}\rho\|^{2}_{L^{2}} &\leq&\end{eqnarray}$$\|\nabla F_{0}\|^{2}_{L^{2}}+\|\rho\|^{2}_{L^{2}}\int_{0}^{t}\|\nabla\rho\|^{2}_{L^{2}} +\int_{0}^{t}\Big(\Big[4+\|E\|^{2}_{L^{2}}+4\|\nabla\rho\|^{2}_{L^{2}}\Big]\|\nabla F\|^{2}_{L^{2}}\Big).$$
So, we have
\begin{eqnarray*}\|\nabla F\|^{2}_{L^{2}}(t)&\leq&\|\nabla F_{0}\|^{2}_{L^{2}}+\|\rho\|^{2}_{L^{2}}\int_{0}^{t}\|\nabla\rho\|^{2}_{L^{2}}+\int_{0}^{t}\Big(\Big[4+\|E\|^{2}_{L^{2}}+4\|\nabla\rho\|^{2}_{L^{2}}\Big]\|\nabla F\|^{2}_{L^{2}}\Big).\end{eqnarray*}
We deduce from Gronwall lemma that
$$\|\nabla F\|^{2}_{L^{2}}(t)\leq\Big(\|\nabla F_{0}\|^{2}_{L^{2}}+\|\rho\|^{2}_{L^{2}}\int_{0}^{t}
\|\nabla\rho\|^{2}_{L^{2}}\Big)e^{\ds\int_{0}^{t}\Big(4+\|E\|^{2}_{L^{2}}+4\|\nabla\rho\|^{2}_{L^{2}}\Big)}$$
for $0<t<T$.
This, formally, yields the bound $\nabla F \in L^\infty(0,T;\; L^2(\mathbb{R}^{2}))$

Moreover, if we return to (\ref{eq}) we get
\begin{eqnarray*}\int_{0}^{t}\|\nabla^{2}\rho\|^{2}_{L^{2}}&\leq&\|\nabla F_{0}\|^{2}_{L^{2}}+\|\rho\|^{2}_{L^{2}}\int_{0}^{t}\|\nabla\rho\|^{2}_{L^{2}}\\\,&\,&+\int_{0}^{t}\Big(\Big[4+\|E\|^{2}_{L^{2}}+4\|\nabla\rho\|^{2}_{L^{2}}\Big]\|\nabla F\|^{2}_{L^{2}}\Big)\end{eqnarray*}
We obtain that $\nabla^{2}\rho\in L^2(0,T;\; L^2)$. Then, we deduce that
$$\nabla\rho\in L^\infty(0,T;\; L^2)\cap L^2(0,T;\; H^1)\mbox{ and }\,\nabla E,\,\nabla B\in L^\infty(0,T;\; L^2).$$
Now, by using the previous section, we have
$$\rho\in L^\infty(0,T;\; H^1)\cap L^2(0,T;\; H^2)\mbox{ and }\, E,\, B\in L^\infty(0,T;\; H^1).$$
If we denote $F= (E,B)$ then we get from Lemma 1.3 that for $t>0$:
$$\|F(t)\|_{H^{1}}\leq\|F_{0}\|_{H^{1}}+\int_{0}^{t}\|j(s)\|_{H^{1}}ds$$
Using that
$$j=\rho E-\nabla\rho\mbox{ and }\nabla j=\nabla\rho.E+\rho\nabla E-\nabla^{2}\rho,$$
 we obtain
 \begin{eqnarray*}\|F(t)\|_{H^{1}}&\leq&\|F_{0}\|_{H^{1}}+Z_{1}(t)+Z_{2}(t)+Z_{3}(t)+Z_{4}(t) .\end{eqnarray*}
Where $$Z_{1}(t)=\|\rho E-\nabla\rho\|_{L^{1}L^{2}},\, Z_{2}(t)=\|\nabla\rho.E\|_{L^{1}L^{2}},$$
 $$Z_{3}(t)=\|\rho\nabla E\|_{L^{1}L^{2}} \mbox{ and } Z_{4}(t)=\|\nabla^{2}\rho\|_{L^{1}L^{2}}.$$
We have from Cauchy-Schwartz inequality that
 \begin{eqnarray*}
Z_{1}(t)&\leq&\int_{0}^{t}\|\rho(s) \|_{L^{\infty}}\|F(s)\|_{H^{1}}ds+\int_{0}^{t}\|\nabla\rho(s) \|_{L^{2}}ds\\
\,&\leq&\int_{0}^{t}\|\rho(s) \|_{L^{\infty}}\|F(s)\|_{H^{1}}ds+t^{\frac{1}{2}}(\int_{0}^{t}\|\nabla\rho(s) \|_{L^{2}}^{2}ds)^{\frac{1}{2}}\\
\,&\leq&\int_{0}^{t}\|\rho(s) \|_{L^{\infty}}\|F(s)\|_{H^{1}}ds+t^{\frac{1}{2}}C_{0}^{\frac{1}{2}}.
\end{eqnarray*}
and we have from Gagliardo-Nirenberg inequality \eqref{GN} that
 \begin{eqnarray*}
Z_{2}(t)&\leq&\int_{0}^{t}\|\nabla\rho(s) \|_{L^{4}}\|E(s)\|_{L^{4}}ds\\
\,&\leq& \int_{0}^{t}\|\nabla\rho(s) \|^{\frac{1}{2}}_{L^{2}}\|\nabla^{2}\rho(s) \|^{\frac{1}{2}}_{L^{2}}\|E(s)\|^{\frac{1}{2}}_{L^{2}}\|\nabla E(s)\|^{\frac{1}{2}}_{L^{2}}ds\\
\,&\leq& \int_{0}^{t}\|\nabla^{2}\rho(s) \|_{L^{2}}ds+\int_{0}^{t}\|\nabla\rho(s) \|_{L^{2}}\|E(s)\|_{L^{2}}\|\nabla E(s)\|_{L^{2}}ds\\
\,&\leq& \int_{0}^{t}\|\nabla^{2}\rho(s) \|_{L^{2}}ds+\|E\|_{L^{\infty}(L^{2})}\int_{0}^{t}\|\nabla\rho(s) \|_{L^{2}}\|F(s)\|_{H^{1}}ds\\
\,&\leq& \int_{0}^{t}\|\nabla^{2}\rho(s) \|_{L^{2}}ds+C_{0}\int_{0}^{t}\|\nabla\rho(s) \|_{L^{2}}\|F(s)\|_{H^{1}}ds.
\end{eqnarray*}
Moreover,
 \begin{eqnarray*}
Z_{3}(t)&\leq&\int_{0}^{t}\|\rho(s) \|_{L^{\infty}}\|F(s)\|_{H^{1}}ds.
\end{eqnarray*}
Combining $Z_{1}, Z_{2}, Z_{3}$ and $Z_{4}$, we get
 \begin{eqnarray*}
\|F(t)\|_{H^{1}}&\leq&\|F_{0}\|_{H^{1}}+t^{\frac{1}{2}}C_{0}^{\frac{1}{2}}+2\int_{0}^{t}\|\rho(s) \|_{L^{\infty}}\|F(s)\|_{H^{1}}ds\\
\,&\,&+2\int_{0}^{t}\|\nabla^{2}\rho(s) \|_{L^{2}}ds+C_{0}\int_{0}^{t}\|\nabla\rho(s) \|_{L^{2}}\|F(s)\|_{H^{1}}ds\\
\,&\leq&\|F_{0}\|_{H^{1}}+C_{0}t^{\frac{1}{2}}\\
\,&\,&+\int_{0}^{t}(2\|\rho(s) \|_{L^{\infty}}+C_{0}\|\nabla\rho(s)\|_{L^{2}})\|F(s)\|_{H^{1}}ds.
\end{eqnarray*}
We deduce from Gronwall lemma that
 \begin{eqnarray*}
\|F(t)\|_{H^{1}}&\leq&\Big(\|F_{0}\|_{H^{1}}+C_{0}t^{\frac{1}{2}}\Big)e^{\ds\int_{0}^{t}(2\|\rho(s) \|_{L^{\infty}}+C_{0}\|\nabla\rho(s)\|_{L^{2}})ds}.
\end{eqnarray*}
Then, using the inequality $( \log(e+ae^{b})\leq  \log(e+a)+b, a\geq0, b\geq0)$, we obtain
 \begin{eqnarray*}
 \log(e+\|F(t)\|_{H^{1}}) \leq \log(e+\|F_{0}\|_{H^{1}}+C_{0}t^{\frac{1}{2}})+C_{0}^{\frac{1}{2}}t^{\frac{1}{2}}+C\|\rho \|_{L^{1}L^{\infty}}.
\end{eqnarray*}
By using inequality \eqref{ineq2} and Lemma 1.6, we obtain
 \begin{eqnarray*}
 \log(e+\|F(t)\|_{H^{1}}) \leq \log(e+\|F_{0}\|_{H^{1}})+C_{0}^{\frac{1}{2}}t^{\frac{1}{2}}\\+C t^{1/2}\|\nabla\rho\|_{L^{2}_{t,x}}\log(e+\frac{\|\nabla^{2}\rho\|_{L^{2}_{t,x}}}{\|\nabla\rho\|_{L^{2}_{t,x}}}).
\end{eqnarray*}
Then, we use that the function $x\rightarrow x\log(e+\frac{C}{x})$ is increasing in $x$ to deduce that there exists a $C_{0}$ such that for all $T>0$, we have
\begin{eqnarray*}
 \sup_{0\leq t\leq T}\log^{\frac{1}{2}}(e+\|F(t)\|_{H^{1}}) \leq \log^{\frac{1}{2}}(e+\|F_{0}\|_{H^{1}})+C_{0}^{\frac{1}{2}}T^{\frac{1}{2}}.
\end{eqnarray*}
Therefore, there exists a constant $C$ such that for all $T>0$, we have
\begin{eqnarray*}
\log(e+\|F(T)\|_{L^{\infty}H^{1}}) \leq C\log(e+\|F_{0}\|_{H^{1}})+C_{0}T.
\end{eqnarray*}
Then, there exists a constant $D_{0}$ depending on $\|F_{0}\|_{H^{1}}$ such that for all $T>0$, we have
 \begin{eqnarray*}
\|F(T)\|_{H^{1}} &\leq&D_{0}(e+\|F_{0}\|_{H^{1}})e^{C_{0}T}.
\end{eqnarray*}

\section{Proof of Theorem 1.1}

In this section, we prove the existence and uniqueness of Theorem 1.1.
\subsection{Existence of solutions}
The existence of a solution $(\rho, E, B)$ which solves (\ref{s1}) follows from the a priori estimates proved in the last section. We shall use the very classical Friedrich’s method (also called Galerkin method in the periodic case) which consists in approximating the system (\ref{s1}) by a cutoff in the frequency system. For this, let us define the regularization operator $J_{n}$ by:
\begin{eqnarray}\label{jn}\forall\, n\in\mathbb{N},\quad   J_{n}u:=\mathcal{F}^{-1}(\mathbf{1}_{B(0,n)}(\xi)\widehat{u}(\xi)),\end{eqnarray}
where the Fourier transform $\mathcal{F}$ in the space variables defined by $$\mathcal{F}(u)(\xi):=\hat{u}(\xi):=\int_{\mathbb{R}^{d}}e^{-ix.\,\xi}u(x)dx.$$

This operator has a regularizing effect since the Plancherel equality allows to write that, for all $s\geq0:$
\begin{eqnarray}\label{hjn}
\exists C>0,\;\forall\, n\in\mathbb{N},\quad  \|J_{n}u\|_{H^{s}(\mathbb{R}^{d})}\leq C (1+n)^{s}  \|u\|_{L^{2}(\mathbb{R}^{d})}. \end{eqnarray}
On the other hand, we have by the Lebesgue theorem that for any $u\in\dot{H}^{s}(\mathbb{R}^{d}),$
\begin{eqnarray}\label{limjn}
\ds\lim_{n\rightarrow\infty} \|J_{n}u-u\|_{\dot{H}^{s}(\mathbb{R}^{d})}=0,
\end{eqnarray}where $$\dot{H}^{s}(\mathbb{R}^{d})=\{u\in L_{loc}^{1}(\mathbb{R}^{d})|\,\,\|u\|_{\dot{H}^{s}}<\infty\}$$ and $$\|u\|_{\dot{H}^{s}}=(\int_{\mathbb{R}^{d}}|\xi|^{2s}|\hat{u}(\xi)|^{2})^{\frac{1}{2}}$$

 Let us consider the approximate system:
\begin{equation}\label{sn}\left\{%
\begin{array}{lll}
\partial_t\rho_{n}+J_{n}div_{x}(J_{n}\rho_{n} J_{n}E_{n}) - \triangle J_{n}\rho_{n}=0,& \; & \mbox{in }(0,\;T)\times\mathbb{R}^2,\\
\partial_{t}E_{n}-curl_{x}J_{n}B_{n}=-j_{n},& \; & \mbox{in }(0,\;T)\times\mathbb{R}^2,\\
\partial_{t}B_{n}+curl_{x}J_{n}E_{n}=0,& \; & \mbox{in }(0,\;T)\times\mathbb{R}^2,\\
div_{x}(J_{n}E_{n})=J_{n}\rho_{n},& \; & \mbox{in }(0,\;T)\times\mathbb{R}^2,\\
div_{x}B_{n}=0,& \; & \mbox{in }(0,\;T)\times\mathbb{R}^2,\\
j_{n}= J_{n}(J_{n}\rho_{n} J_{n}E_{n}-\nabla_xJ_{n}\rho_{n}),& \; & \\
\end{array}%
\right.\end{equation}
with the following initial data
$$\rho_{n}(t=0)=J_{n}(\rho_0),\,\,B_{n}(t=0)=J_{n}(B_0),\,\,E_{n}(t=0)=J_{n}(E_0).$$
The above system appears as a system of ordinary differential equations on $$L^{2}_{n}=\{u\in L^{2}(\mathbb{R}^{2})|\;J_{n}u=u\}.$$
Indeed, by using (\ref{hjn}) we can transform (\ref{sn}) as a system of ordinary differential equations $$\partial_{t}u_{n}=F(u_{n})$$
where $$u_{n}=(\rho_{n}, E_{n}, B_{n})$$ and $$F(u_{n})=(J_{n} \triangle J_{n}\rho_{n} -J_{n}div_{x}(J_{n}\rho_{n} J_{n}E_{n}), curl_{x}J_{n}B_{n}-j_{n}, -curl_{x}J_{n}E_{n})$$
So we remark that
\begin{eqnarray*}
\|J_{n} \triangle J_{n}\rho_{n}\|_{L^{2}(\mathbb{R}^{2})}\leq C (1+n)^{2}  \|\rho_{n}\|_{L^{2}(\mathbb{R}^{2})}\\
\leq C (1+n)^{2}  \|u_{n}\|_{L^{2}(\mathbb{R}^{2})}
\end{eqnarray*}
by using (\ref{hjn}) and the embedding $H^{s}(\mathbb{R}^{d})\hookrightarrow L^{\infty}(\mathbb{R}^{d})$ for $s>d/2$, we get

\begin{eqnarray*}
\|J_{n}div_{x}(J_{n}\rho_{n} J_{n}E_{n})\|_{L^{2}(\mathbb{R}^{2})} &\leq& C (1+n)\| J_{n}\rho_{n} J_{n}E_{n}\|_{L^{2}(\mathbb{R}^{2})}\\
&\leq& C (1+n) \|J_{n}\rho_{n}\|_{L^{\infty}(\mathbb{R}^{2})}\| J_{n}E_{n}\|_{L^{2}(\mathbb{R}^{2})}\\
&\leq& C_{n} \|\widehat{J_{n}\rho_{n}}\|_{L^1(\mathbb{R}^{2})}\| E_{n}\|_{L^{2}(\mathbb{R}^{2})}\\
&\leq& C_{n} \|\rho_{n}\|_{L^{2}(\mathbb{R}^{2})}\|E_{n}\|_{L^{2}(\mathbb{R}^{2})}\\
&\leq& C_{n} \|u_{n}\|_{L^{2}(\mathbb{R}^{2})}^{2}
\end{eqnarray*}

\begin{eqnarray*}\|curl_{x}J_{n}B_{n}\|_{L^{2}(\mathbb{R}^{2})} &\leq& C (1+n) \ds\sum_{i,j}\|\partial_{i} J_{n}B_{n}^{j}\|_{L^{2}(\mathbb{R}^{2})} \\
&\leq& C (1+n) \| J_{n}B_{n}\|_{H^{1}(\mathbb{R}^{2})}\\
&\leq& C_{n}  \|B_{n}\|_{L^{2}(\mathbb{R}^{2})}\\
&\leq& C_{n}\|u_{n}\|_{L^{2}(\mathbb{R}^{2})}
\end{eqnarray*}
$\|J_{n}(J_{n}\rho_{n} J_{n}E_{n}-J_{n}\nabla J_{n}\rho_{n})\|_{L^{2}(\mathbb{R}^{2})}\leq C \|J_{n}\rho_{n} J_{n}E_{n}\|_{L^{2}(\mathbb{R}^{2})}$
\begin{eqnarray*}
\quad&\,&+C \|\nabla J_{n}\rho_{n}\|_{L^{2}(\mathbb{R}^{2})} \\
&\leq& C_{n}\|J_{n}\rho_{n}\|_{L^{\infty}(\mathbb{R}^{2})}\| J_{n}E_{n}\|_{L^{2}(\mathbb{R}^{2})}+C_{n}\| J_{n}\rho_{n}\|_{H^{1}(\mathbb{R}^{2})}\\
&\leq&  C_{n} \|\rho_{n}\|_{L^{2}(\mathbb{R}^{2})}\|E_{n}\|_{L^{2}(\mathbb{R}^{2})}+ C_{n}\|\rho_{n}\|_{L^{2}(\mathbb{R}^{2})}\\
&\leq&  C_{n}\|\rho_{n}\|_{L^{2}(\mathbb{R}^{2})}(\| E_{n}\|_{L^{2}(\mathbb{R}^{2})}+1) \\
&\leq&  C_{n}\|u_{n}\|_{L^{2}(\mathbb{R}^{2})}(\|u_{n}\|_{L^{2}(\mathbb{R}^{2})}+1) \\
\end{eqnarray*}

So, the usual Cauchy-Lipschitz theorem implies the existence of a unique solution $u_{n}$ of (\ref{sn}) which is in $C^{1}([0, T_{n}[; L^{2}_{n})$ and a strictly positive maximal time $T_{n}$ which verify:
\begin{eqnarray}\label{tn}T_{n}<\infty\quad \Rightarrow \quad \|u_{n}(T_{n})\|_{L^{2}_{n}}=\infty\end{eqnarray}

 But, as $J_{n}^{2}=J_{n}$, we claim that $J_{n}u_{n}$ is also a solution, so uniqueness implies that $J_{n}u_{n}=u_{n}$ and hence, one can remove all the $J_{n}$ in front of $\rho_{n}$, $B_{n}$ and $E_{n}$ keeping
only those in front of nonlinear terms:
\begin{equation}\label{snn}\left\{%
\begin{array}{lll}
\partial_t\rho_{n}+J_{n}div_{x}(\rho_{n}E_{n}) -J_{n} \triangle \rho_{n}=0,& \; & \mbox{in }(0,\;T)\times\mathbb{R}^2,\\
\partial_{t}E_{n}-curl_{x}B_{n}=-j_{n},& \; & \mbox{in }(0,\;T)\times\mathbb{R}^2,\\
\partial_{t}B_{n}+curl_{x}E_{n}=0,& \; & \mbox{in }(0,\;T)\times\mathbb{R}^2,\\
div_{x}E_{n}=\rho_{n},& \; & \mbox{in }(0,\;T)\times\mathbb{R}^2,\\
div_{x}B_{n}=0,& \; & \mbox{in }(0,\;T)\times\mathbb{R}^2,\\
j_{n}= J_{n}(\rho_{n} E_{n}-\nabla_x\rho_{n}),& \; & \\
\end{array}%
\right.\end{equation}
The main goal is to prove that $T_{n}$ can be taken to be equal to $+\infty$ and that we have some local in time estimate which are uniform in $n$. Then, one can pass to the limit and recover a solution of the initial system (\ref{s1}).

As $J_{n}$ is a Fourier multiplier, it commutes with constant coefficient differentiations and hence, the energy estimate (\ref{ident}) still holds:
 \begin{eqnarray}\label{identn}\frac{1}{2}[\|\rho_{n}\|^{2}_{L^{2}}+\|E_{n}\|^{2}_{L^{2}}+\|B_{n}\|^{2}_{L^{2}}](t)
 +\int_{0}^{t}\|\nabla\rho_{n}\|^{2}_{L^{2}}\\\nonumber
\leq\frac{1}{2}[\|J_{n}(\rho_{0})\|^{2}_{L^{2}}+\|J_{n}(E_{0})\|^{2}_{L^{2}}+\|J_{n}(B_{0})\|^{2}_{L^{2}}]\\\nonumber
\leq\frac{1}{2}[\|\rho_{0}\|^{2}_{L^{2}}+\|E_{0}\|^{2}_{L^{2}}+\|B_{0}\|^{2}_{L^{2}}]\leq C_{0}. \end{eqnarray}
This implies that by (\ref{tn}) the $L^{2}$ norm of $(\rho_{n}, B_{n}, E_{n})$ is controlled and hence, $T_{n} =+\infty$ for all $n\in\mathbb{N}$. Moreover, the estimates performed in the previous section apply in the same way to the system (\ref{snn}) and hence the a priori estimates derived there still hold (with bounds which are independent of n), namely we have:
 \begin{eqnarray}
\|F_{n}(t)\|_{H^{1}} \leq C(e+\|F_{0}\|_{H^{1}})e^{C_{0}t}
\end{eqnarray}
and
\begin{eqnarray}
\|\rho_{n}\|_{L^{1}H^{1}}\leq C (e+\|F_{0}\|_{H^{1}})e^{C_{0}T}.
\end{eqnarray}
Now, using that $E_{n},\;B_{n}\in L^{\infty}(0,T; L^{2})$ we can see that $\partial_{t}B_{n}=curl E_{n}$ is bounded in $L^{\infty}(0,T; H^{-1})$ and $\partial_{t}E_{n}=curl B_{n}-j_{n}$ is bounded in $L^{\infty}(0,T; H^{-1})+L^{1}H^{1}$. We also have, $\partial_{t}\rho_{n}=J_{n}\triangle\rho_{n}-J_{n}div (\rho_{n}E_{n})$. Since we know that $\rho_{n}$ is bounded in $L^{2}(0,T; \dot{H}^{1})$ then $\triangle\rho_{n}$ is bounded in $L^{2}(0,T; \dot{H}^{-1})$ and recall that $\rho_{n}$ is bounded in $L^{2}(0,T; L^{\infty})$ and $E_{n}$ is bounded in $L^{\infty}(0,T; L^{2})$ so due to lemma \ref{rl1} we get that for all $T >0$, there exists a constant $C_{T}$ such that
\begin{eqnarray}
\|\partial_{t}\rho_{n}\|_{L^{2}(0,T; H^{-1})}\leq C_{T}\mbox{ and } \|\partial_{t}F_{n}\|_{L^{2}(0,T; H^{-1})}\leq C_{T}.
\end{eqnarray}
Hence, extracting a subsequence, standard compactness arguments allow us to pass to the limit in (\ref{snn}). This yields the existence of a solution $(\rho, E, B)$ to (\ref{s1}) (see for instance \cite{Mas}) with the initial data (\ref{d1}).
\subsection{Uniqueness of solutions}
Here, we prove the uniqueness of solutions to (\ref{s1}) in $L^{\infty}(0,T; L^{2})\cap L^{2}(0,T; H^{1})\times L^{\infty}(0,T; H^{1})\times L^{\infty}(0,T; H^{1})$.
Actually, we prove here a uniqueness result slightly stronger than the one stated in the theorem since we do not require the continuity in time. This actually is a very small improvement since one can get the continuity just from the fact that $(\rho_{i},E_{i},B_{i} )$ solves the system. Take $(\rho_{1},E_{1},B_{1})$ and $(\rho_{2},E_{2},B_{2})$ two solutions of (\ref{s1}) with the same initial condition (\ref{d1}) and such that for $i = 1, 2$, we have $\rho_{i}\in L^{\infty}(0,T; L^{2})\cap L^{2}(0,T; H^{1})$ and $E_{i}, B_{i} \in L^{\infty}(0,T; H^{1})$.
We denote $\rho = \rho_{2} -\rho_{1}$, $E = E_{2} -E_{1}$, $j = j_{2} - j_{1}$ and $B = B_{2 }- B_{1}$. We have:
\begin{equation}\label{u1}\left\{%
\begin{array}{lll}
\partial_t\rho+\nabla_{x}\rho\,E+\rho^{2}-\triangle\rho =0,& \; & \mbox{in }(0,\;T)\times\mathbb{R}^2,\\
\partial_{t}E-curl_{x}B=-j,& \; & \mbox{in }(0,\;T)\times\mathbb{R}^2,\\
\partial_{t}B+curl_{x}E=0,& \; & \mbox{in }(0,\;T)\times\mathbb{R}^2,\\
div_{x}E=\rho,& \; & \mbox{in }(0,\;T)\times\mathbb{R}^2,\\
div_{x}B=0,& \; & \mbox{in }(0,\;T)\times\mathbb{R}^2,\\
j= \rho E-\nabla_x\rho,& \; & \\
\end{array}%
\right.\end{equation}
We denote $X =L^{\infty}(0,T; L^{2})\cap L^{2}(0,T; H^{1})$. We also denote $Y = X\times L^{\infty}(0,T; H^{1})\times L^{\infty}(0,T; H^{1})$ and use that $X\subset L^{2}(0,T; H^{1})$.\\

Multiplying the first equation of (\ref{u1}) by $\rho$ and integrating by parts,
\begin{eqnarray*}
 \frac{1}{2}\partial_{t}\|\rho\|^{2}_{L^{2}}+\|\nabla\rho\|^{2}_{L^{2}}&=& -\int\rho\nabla_{x}\rho\,E dx-\int\rho^{3}dx\\
 &=& -\int\nabla_{x}(\frac{\rho^{2}}{2})\,E dx-\int\rho^{3}dx\\
  &=& \frac{1}{2}\int\rho^{3} dx-\int\rho^{3}dx\\
   &=& -\frac{1}{2}\int\rho^{3} dx\\
&\leq&\frac{1}{2}\|\rho\|_{L^{\infty}}\|\rho\|_{L^{2}}^{2}. \end{eqnarray*}
Now, integrating the result
with respect to t. Thus we obtain
\begin{eqnarray}\label{ro}
\|\rho(t)\|^{2}_{L^{2}}+2\int_{0}^{t}\|\nabla\rho\|^{2}_{L^{2}}d\tau \leq\|\rho_{0}\|^{2}_{L^{2}}+\int_{0}^{t}\|\rho\|_{L^{\infty}}\|\rho\|_{L^{2}}^{2} d\tau.
 \end{eqnarray}
Hence, by Gronwall lemma, we obtain
\begin{eqnarray}\label{ro1}
\|\rho(t)\|^{2}_{L^{2}} \leq\|\rho_{0}\|^{2}_{L^{2}}\exp(\int_{0}^{t}\|\rho\|_{L^{\infty}} d\tau)
 \end{eqnarray}
that $\|\rho(t)\|_{L^{2}}=0$ since $\rho_{0}=0$. Therefore, we deduce from (\ref{ro}), we get $\|\rho\|_{X}=0$. On the other hand, we have,
 \begin{eqnarray*}
 \|F\|_{L^{\infty}H^{1}}&\leq& C \|j\|_{L^{1}H^{1}}\\
\,&\leq& C \|\rho E\|_{L^{1}H^{1}}+ C\|\nabla_x\rho\|_{L^{1}H^{1}}\\
\,&\leq&C \|\rho\|_{L^{1}L^{\infty}}\| E\|_{L^{\infty}H^{1}}+ C\|\nabla_x\rho\|_{L^{1}L^{2}}+C\|\nabla_x^{2}\rho\|_{L^{1}L^{2}}\\
 \,&\leq&C \|\rho\|_{L^{1}L^{\infty}}\| F\|_{L^{\infty}H^{1}}+ C\|\rho\|_{L^{1}H^{1}}+C T^{1/2}\|\nabla_x^{2}\rho\|_{L^{2}_{t,x}}\\
  \,&\leq&C T^{1/2}\| F\|_{L^{\infty}H^{1}}+ C T^{1/2}\|\rho\|_{L^{2}H^{1}}+C T^{1/2}\|\nabla_x^{2}\rho\|_{L^{2}_{t,x}}\\
    \,&\leq&C T^{1/2}\| F\|_{L^{\infty}H^{1}}+ C T^{1/2}\|\rho\|_{X}+C T^{1/2}\|\nabla_x^{2}\rho\|_{L^{2}_{t,x}}
  \end{eqnarray*}

Choosing T small enough, we get $F = 0$, which yields the uniqueness of the solution on a small time interval. One can then repeat the argument and get the uniqueness on the whole real line.

\noindent{\bf Acknowledgments.} {\it  We are very grateful  to N. Masmoudi for  interesting
 discussions around the questions dealt with in this paper.}



\begin{thebibliography}{10}


\bibitem{Amirat} Y. Amirat, K. Hamdache, F. Murat, {\em Global weak solutions to equations of motion for magnetic fluids}, J. Math. Fluid Mech., {\bf (10:3)}, 326--351, 2008.

\bibitem{A T2004}N. Ben Abdallah, M. L. Tayeb, {\em Diffusion limit
for the one dimensional Boltzmann-Poisson system},  Discete.
Contin. Dyn. Syst. Ser. B, 1129--1142, 2004.

\bibitem{Bony} J.-M. Bony, {\em Calcul symbolique propagation des singularités pour les \'equations aux deriv\'ees partielles non lin\'eaires}, Ann. Sci. $\acute{\mbox{E}}$cole Norm. Sup. {\bf(14:2)}, 209--246, 1981.

\bibitem{gaj} H. Gajewski, {\em  On uniqueness of solutions of the Drift-Diffusion-Model of semiconductors devices}, Math. Models Methods Appl. Sci., {\bf (4)}, 121--139, 1994.


\bibitem{bostan goudon} M. Bostan, T. Goudon, {\em  Low field regime for the relativistic Vlasov-Maxwell-Fokker-Planck system; the one and one half dimensional case}, Kinetic related models, {\bf (1:1)}, 139--196, 2008.

  \bibitem{Constantin} P. Constantin, N.Masmoudi, {\em Global well-posedness for a Smoluchowski equation coupled with Navier.Stokes equations in 2D}, Comm. Math.
Phys., {\bf (278:1)}, 179--191, 2008.

\bibitem{gh1} N. El Ghani, {\em Diffusion limit for the {V}lasov-{M}axwell-{F}okker-{P}lanck system},
 IAENG International Journal of Applied Mathematics, {\bf (40:3)}, 159--166, 2010.

 \bibitem{gh} N. El Ghani, N. Masmoudi, {\em Diffusion limit for the {V}lasov-{P}oisson-{F}okker-{P}lanck system},
 Comm. Math. Sci.,  {\bf (8:2)}, 463--479, 2010.

\bibitem{goudon} T. Goudon, {\em Hydrodynamic limit for the Vlasov-Poisson-Fokker-Planck system:
analysis of the two-dimentional case}, Math. Mod. Meth. Appl. Sci, {\bf (15)}, 737--752, 2005.



\bibitem{Ibrahim} S. Ibrahim, M. Majdoub, N. Masmoudi, {\em Global solutions for a semilinear two-dimensional Klein.Gordon equation with exponential-type nonlinearity}, Comm. Pure Appl. Math., {\bf (59:11)}, 1639--1658, 2006.

\bibitem{Ibrahim2} S. Ibrahim, M. Majdoub, N. Masmoudi, {\em Double logarithmic inequality with a sharp constant}, Proc. Amer. Math. Soc., {\bf (135:1)}, 87--97, 2007.(electronic).

 \bibitem{Lin} F. Lin, P. Zhang, Z. Zhang, {\em On the global existence of smooth solution to the 2-D FENE dumbbell model}, Comm. Math. Phys., {\bf (277:2)}, 531--553, 2008.

\bibitem{Lions} P.-L. Lions, N. Masmoudi, {\em Global solutions for some Oldroyd models of non-Newtonian flows}, Chinese Ann. Math. Ser. B, {\bf (21:2)}, 131--146, 2000.

\bibitem{Lions2} P.-L. Lions, N. Masmoudi, Global existence of weak solutions to some micro.macro models, C. R. Math. Acad. Sci. Paris, {\bf (345:1)}, 15--20, 2007.


\bibitem{Mas} N. Masmoudi, {\em Global well posedness for the {M}axwell-{N}avier-{S}tokes
              system in 2{D}}, J. Math. Pures Appl. {\bf (9)}, 559--571, 2010.

\bibitem{Mas2} N. Masmoudi, {\em Well-posedness for the FENE dumbbell model of polymeric flows}, Comm. Pure Appl. Math., {\bf (61:12)}, 1685--1714, 2008.

\bibitem{m-t} N. Masmoudi, M. L. Tayeb, {\em Diffusion limit of a
semiconductor Boltzmann-Poisson system}, SIAM J. Math. Anal, {\bf(36:6)}, 1788--1807, 2007.


\bibitem{Mas3} N. Masmoudi, P. Zhang, Z. Zhang, {\em Global well-posedness for 2D polymeric fluid models and growth estimate}, Physica D 237, {\bf (10:12)},
1663--1675, 2008.

\bibitem{ps}F. Poupaud, J. Soler, {\em Parabolic limit and stability of
the Vlasov-Fokker-Planck system}, Math. Models Methods Appl. Sci, {\bf(10:7)}, 1027--1045, 2000.
	




\end{thebibliography}
\end{document}